\documentclass{amsart}
\usepackage{amsmath,amsthm,amscd,amssymb, color}
\usepackage{latexsym}
\usepackage[colorlinks, citecolor = blue]{hyperref}
\usepackage[alphabetic]{amsrefs}
\usepackage[capitalize, nameinlink]{cleveref}
\usepackage{enumerate}

\newcommand{\R}{\mathbb R}

\newcommand{\Q}{\mathbb Q}
\newcommand{\C}{\mathbb C}
\newcommand{\Z}{\mathbb Z}
\newcommand{\bfk}{\mathbf k}
\renewcommand{\phi}{\varphi}
\newcommand{\eps}{\varepsilon}
\newcommand{\frakp}{\mathfrak p}
\newcommand{\frakq}{\mathfrak q}
\newcommand{\frakN}{\mathfrak N}
\newcommand{\frakM}{\mathfrak M}
\newcommand{\frako}{\mathfrak o}
\newcommand{\calO}{\mathcal O}
\newcommand{\calM}{\mathcal M}
\newcommand{\calS}{\mathcal S}
\newcommand{\calE}{\mathcal E}
\newcommand{\calI}{\mathcal I}
\newcommand{\A}{\mathbb A}
\newcommand{\bs}{\backslash}
\newcommand{\bmx}{\left( \begin{matrix}}
\newcommand{\emx}{\end{matrix} \right)}
\newcommand{\Cl}{\mathrm{Cl}}
\newcommand{\new}{\mathrm{new}}
\newcommand{\odd}{\mathrm{odd}}
\newcommand{\adm}{\mathrm{adm}}
\renewcommand{\mod}{\, \, \mathrm{mod} \, \,}

\usepackage{bbm}
\newcommand{\one}{\mathbbm{1}}
\newcommand{\two}{\mathbf{2}}
\newcommand{\zero}{\mathbf{0}}

\newcommand{\leg}{\overwithdelims ()}

\DeclareMathOperator{\GL}{GL}
\DeclareMathOperator{\tr}{tr} 

\newtheorem{lem}{Lemma}
\numberwithin{lem}{section}
\newtheorem{prop}[lem]{Proposition}
\newtheorem{thm}[lem]{Theorem}

\newtheorem{cor}[lem]{Corollary}

\theoremstyle{remark}
\newtheorem{rem}{Remark}
\numberwithin{rem}{section}

\theoremstyle{definition}

\numberwithin{equation}{section}

\begin{document}

\title{Congruences for modular forms mod 2 and quaternionic $S$-ideal classes}
\author{Kimball Martin}
\address{Department of Mathematics, University of
Oklahoma, Norman, OK 73019}

\date{\today}

\maketitle

\begin{abstract}
We prove many simultaneous congruences mod 2 for elliptic and Hilbert modular forms 
among forms with different Atkin--Lehner eigenvalues. The proofs involve the notion of quaternionic $S$-ideal classes and the distribution of Atkin--Lehner signs among
newforms.
 \end{abstract}

%
%

\section{Introduction}

In this paper, we use the notion of quaternionic $S$-ideal classes
and the Jacquet--Langlands correspondence to show certain behavior of
Atkin--Lehner signs yields many simultaneous congruences of newforms mod 2.
We begin by explaining our main results over $\Q$, and will discuss 
the extensions to Hilbert modular forms at the end of the introduction.

Let $N$ be a squarefree product of an odd number of primes, $M | N$ and $k \in 2 
\mathbb N$.
By a sign pattern $\eps$ for $M$ we mean a collection of signs $\eps_p \in \{ \pm 1 \}$ 
for each $p | M$.  
Denote the
sign pattern with $\eps_p = -1$ for all $p | M$ by $-_M$.

Let $S_k^\new(N)$ denote the span of newforms of weight $k$
for $\Gamma_0(N)$.  For a sign pattern $\eps$ for $M$, let $S_k^{\new,\eps}(N)$ be 
the subspace spanned by newforms $f$
with $p$-th Atkin--Lehner eigenvalue $w_p(f) = \eps_p$ for all $p | M$.  
The case $k=2$ is a little different than $k \ge 4$, due to the interaction with
the weight 2 Eisenstein series $E_{2,N}(z) := \sum_{d|N} \mu(d) d E_2(dz)$.  
To state our first result uniformly,
we introduce the augmented space $S_k^\new(N)^*$, which is just $S_k^\new(N)$
if $k \ge 4$ but $S_2^\new(N)^* = S_2^\new(N) \oplus \C E_{2,N}$.  Similarly,
we set $S_k^{\new, \eps}(N)^* = S_k^{\new, \eps}(N)$ if $k \ge 4$ or $\eps \ne -_M$
and $S_2^{\new, -_M}(N)^* = S_2^{\new, -_N}(N) \oplus \C E_{2,N}$.

Denote the $n$-th Fourier coefficient of a modular form $f$ by $a_n(f)$.  Our first main result is

\begin{thm} \label{thm1}
Suppose $M, N$ are as above such that, for each divisor $d | M$ with $d > 1$, there exists an odd
prime $p | \frac NM$ such that ${-d \leg p} = 1$.  If $M$ is even, assume also
$\frac NM$ is divisible by a prime which is $1 \mod 4$.  
Let $f \in S_k^\new(N)^*$ be a newform and $\ell$ a prime of $\bar \Q$ above $2$.  

Then for any sign pattern $\eps$ for $M$, there exists an eigenform 
$g \in S_k^{\new,\eps}(N)^*$ such that $a_n(f) \equiv a_n(g) \mod \ell$ for all $n \in \mathbb N$.
Moreover, we may take $g \in S_k^{\new}(N)$ to be a cuspidal newform if $k \ne 2$,
$\eps \ne -_M$, or $N$ is not an even product of $3$ primes.
\end{thm}

In particular this theorem applies if there exists a prime $p | \frac NM$ with 
$p \equiv 1 \mod 4$ such that ${q \leg p} = 1$ for each prime $q | M$, e.g.\ $M=26$
and $N=17 \cdot M$.
The theorem also holds with the alternative hypotheses
that $M \equiv 7 \mod 8$ is prime and $N$ is even (see \cref{rem:41}).

Quadratic reciprocity implies that if $p_1 p_2 | N$ and $p_1 \equiv 3 \mod 4$, then
$M=p_1$ or $M=p_2$ satisfies the hypothesis of this theorem.  This yields

\begin{cor} Suppose $N$ is composite and divisible by some $p \equiv 3 \mod 4$.  
Let $k \in 2 \mathbb N$ with $k \ne 2$ if $N=2p_1p_2$ for some primes $p_1, p_2$.  
Fix a prime $\ell$ of $\bar \Q$
above $2$.  Then for any newform $f \in S_k^\new(N)$, there exists a non-Galois-conjugate
newform $g \in S_k^\new(N)$ such that $a_n(f) \equiv a_n(g) \mod \ell$ for all 
$n \in \mathbb N$.
\end{cor}

Our second main result is 

\begin{thm} \label{thm2}
Let $f \in S_2^{\new}(N) \oplus \C E_{2,N}$ be an eigenform, and $\ell$ a prime of
$\bar \Q$ above $2$.  Then there exists
an eigenform $g \in S_2^{\new, -_N}(N) \oplus \C E_{2,N}$ such that
$a_n(f) \equiv a_n(g) \mod \ell$ for all $n \in \mathbb N$.  Moreover,
if $N$ is not an even product of $3$ primes, we may take $g \in S_2^{\new, -_N}(N)$.
\end{thm}

Note many existing congruence results exclude small primes or primes dividing 
the level,  e.g.\ \cite{mazur}, \cite{hida} and \cite{yoo},
whereas our method is specific to congruences mod 2 and does not require
$2 \nmid N$.
Moreover, these results indicate that congruences modulo (primes above) 2 are very 
common.  Indeed, they seem much more common than congruences modulo odd 
primes appear to occur, since (at least large) congruence primes must divide 
the special value of an $L$-function (e.g., see \cite{hida}).  
Further, while many congruence results are known, simultaneous congruence
results seem harder to come by.  However, our results exhibit many simultaneous 
congruences.  

Namely, if $\omega(M)$ is the number
of prime factors of $M$, \cref{thm1} gives conditions for all newforms as well as
$E_{2,N}$ to be congruent to at least $2^{\omega(M)}$ (non-Galois-conjugate) eigenforms.  Further, \cref{thm2}
says that for any squarefree level $N$ with $\omega(N)$ odd, there are at most
$1 + \dim S_2^{\new, -_N}(N)$ congruence classes in $S_2^{\new}(N)$.  An exact
formula for $\dim S_2^{\new, -_N}(N)$ is given in \cite{me:dim} and is approximately
$2^{-\omega(N)} \dim S_2^\new(N)$, so there must be at least one congruence class
containing many newforms when $\omega(N)$ is large.

In weight 2, we note there have been some recent results giving simultaneous
congruences.  Le Hung and Li \cite{lehung-li}, in their investigations on level raising mod 
2, have shown for certain forms
in $S_2(N)$ one gets congruences with forms with prescribed Atkin--Lehner signs.
Specifically, under the assumption that $f$ is not congruent to an Eisenstein series mod 2, the methods in \cite{lehung-li} give both \cref{thm2}, and a version of \cref{thm1} (at 
least in weight 2) where one can prescribe all but one Atkin--Lehner sign for $g$.
We note that our methods seem quite different, though both make use of the Jacquet--Langlands correspondence.

On the other hand, Ribet and later Yoo (see \cite{yoo}) have investigated
congruences of newforms in $S_2(N)$ with Eisenstein series modulo primes $p > 3$
with prescribed Atkin--Lehner signs, which gives simultaneous Eisenstein 
congruences under certain conditions.  Taking $f=E_{2,N}$ in \cref{thm1} gives an 
analogue of the sufficient conditions in \cite{yoo} for an Eisenstein congruence with 
prescribed Atkin--Lehner signs mod 2.  However we cannot specify signs at all places 
(except when all signs are $-1$ by taking $f=E_{2,N}$ in \cref{thm2}).

\medskip
Now let us discuss the proofs, which have a couple of features we find interesting,
such as the connection with distributions of Atkin--Lehner signs and the
connection between certain eigenspaces of quaternionic modular forms and
quaternionic $S$-ideal class.

First we discuss the distribution of Atkin--Lehner signs. 
Let us say the sign patterns $\eps$ for $M$ are perfectly equidistributed in weight 
$k$ and level $N$ if $\dim S_k^{\new, \eps}(N)^*$ is independent of $\eps$.  We 
will find that perfect equidistribution
in weight 2 implies perfect equidistribution in weight $k$.  (This is also evident from
\cite{me:dim} under the hypotheses of \cref{thm1}.)  Then we will prove
that this perfect equidistribution implies the congruences in \cref{thm1}, and
use \cite{me:dim} to see the above hypotheses are sufficient for perfect 
equidistribution.

\cref{thm2} is related to a different fact about distribution of sign patterns.  
In \cite{me:dim}, we showed that although the sign patterns are equidistributed
asymptotically as the weight or level grows, there is a bias toward or against
certain sign patterns in fixed spaces $S_k^\new(N)$.  In particular, when $k=2$
and $\omega(N)$ is odd, there is a bias towards $-_N$ in the sense that
$\dim S_2^{\new, \eps}(N) \le \dim S_2^{\new, -_N}(N)^*$ for any sign pattern $\eps$
for $N$.  Below we will give a simple proof of this using quaternion algebras, and
the idea behind this proof is what allows us to construct the congruences
in \cref{thm2}.

The overall strategy to get our theorems is the use of arithmetic of definite quaternion 
algebras to construct congruences between quaternionic modular forms, and then use the
Jacquet--Langlands correspondence to deduce congruences for elliptic or
Hilbert modular forms.  This is why we restrict to $\omega(N)$ odd over $\Q$.
We also used this idea in \cite{me:cong} to get Eisenstein congruences in weight 2,
generalizing results from \cite{mazur} and \cite{yoo}.  Whereas in that paper we
used mass formulas for quaternionic orders to get Eisenstein congruences, here we use 
the structure of quaternionic $S$-ideal classes to get our congruences.

In \cref{sec:ideals}, we define the notion of $S$-ideal classes for
quaternion algebras in an analogous way to the definition of
$S$-ideal classes for number fields.  The $S$-ideal class numbers interpolate
between the usual class number and the type number of a quaternion algebra.

In \cref{sec:qmfs} we review the theory of (definite) quaternionic modular forms.
If $B$ is a definite quaternion algebra of discriminant $N$ and
$\calO$ is a maximal order of $B$, then the space
$S_{k}^\new(N)^*$ corresponds to a space of $\calM_{k-2}(\calO)$
of quaternionic modular forms.  These can be viewed as certain 
vector-valued functions on the set of right $\calO$-ideal classes $\Cl(\calO)$. 
In the case $k=2$, $S_2^\new(N)^*$ simply 
corresponds to the space of all $\C$-valued functions on $\Cl(\calO)$.

In \cref{sec:action} we describe the action of ramified Hecke operators
on quaternionic modular forms in terms of local involutions acting on $\Cl(\calO)$.
This gives a realization of the space of quaternionic forms corresponding to
$S_k^{\new, \eps}(N)^*$ as certain functions on the set of $S$-ideal classes
$\Cl_S(\calO)$ for $\calO$.  However, the precise description of this space
in general is somewhat complicated as it involves both the way the local
involutions for different primes interact globally as well as the way they
interact with the weight and the signs $\eps_p$.

There are two situations where we can make this description simpler.  One is
if the local involutions act on $\Cl(\calO)$ both without fixed points as well
as without fixing orbits of points under the other local involutions.  This corresponds
to the $S$-class numbers being as small as possible, which corresponds
to perfect equidistribution of sign patterns in weight $2$.   From our description
of quaternionic forms corresponding to $S_k^{\new,\eps}(N)$,
we can deduce that perfect equidistribution in weight 2 implies it in all weights.
In this situation, this is enough construct the quaternionic congruences which imply
\cref{thm1}, excluding the cuspidal condition in weight 2, upon applying our 
 dimension formulas for $S_k^{\new,\eps}(N)$ in \cite{me:dim} to determine when we have 
 perfect equidistribution of signs.

The other situation where this description becomes simpler is in weight 2, so
one only needs to understand how the local involutions interact.
Namely, if $k=2$, these forms are 
just the $\C$-valued functions on the set of $S$-ideal classes which are
``admissible for $-\eps$.''  Since all $S$-ideal classes are admissible when
$\eps=-_N$, this immediately gives bias towards the sign pattern $-_N$
in weight 2.  This description also yields relations between type numbers or generally
$S$-ideal class numbers and dimensions of of spaces of newforms, and allows
us to construct the quaternionic congruences needed for the first part of \cref{thm2}.

To show one can take  $g$ to be a cusp form in \cref{thm2} (and thus also
\cref{thm1}) when $N$ is not an even product of 3 primes, we prove two auxiliary
results.  By a variant of our argument in \cite{me:cong}, we show in \cref{prop:54} that 
$E_{2,N}$ is congruent to a newform in $S_2^{\new,-_N}(N)$ under certain 
conditions; in particular if $\omega(N) > 3$ or $N$ is a product of 3 odd primes.
We treat the $N$ prime case by showing that lack of perfect equidistribution of Atkin--Lehner
signs means the congruent quaternionic modular form we construct must be
cuspidal (\cref{prop:55}).  Using dimension formulas from \cite{me:dim}, we see that
lack of perfect equidistribution is automatic for $N$ prime.  
These auxiliary results in fact give other conditions when
$N=2p_1p_2$ where we can still take $g$ cuspidal in \cref{thm2}---for instance, if
$p_1$ or $p_2$ is $1 \mod 4$ and $N > 258$.  See \cref{sec:eis} for details.
We note that some exceptions to taking $g$ cuspidal when $N=2p_1 p_2$
are in fact necessary, e.g., $S_2^{\new}(42)$ and $S_2^\new(70)$ are 1-dimensional
but not all Atkin--Lehner operators act by $-1$.

Now we summarize what we can say in the case of Hilbert modular forms.
For simplicity, we only work over totally real fields $F$ of narrow class number
$h_F^+=1$, however we expect that our arguments can be suitably modified
to remove this restriction.  (See \cref{sec:qmfs} for comments on what needs to be 
modified.)  The proofs we have described above then go through for Hilbert
modular forms over $F$ with the exception of the explicit determination of
when we have perfect equidistribution of signs, as we have not worked out
an analogue of \cite{me:dim} over totally real fields.  In other words, one
does not have the explicit criteria in terms of quadratic residue symbols for
the Hilbert analogue of \cref{thm1} (see \cref{cor1}), nor does one have
exactly analogous conditions on the level for when one can take $g$ cuspidal in
the analogue of \cref{thm2} (see \cref{cor2}).  However, we can still give some
conditions on when we can take $g$ cuspidal in \cref{cor2} by \cref{prop:54}
which gives Eisenstein congruences under certain hypotheses.

Lastly, we remark in \cite{me:cong} we worked with quaternionic orders which were
not necessarily maximal (or even Eichler), which allowed us to get Eisenstein congruences
for any level $N$ which is not a perfect square, though we could not always say
the congruent cusp form is new.  We expect that our basic strategy here should be
generalizable to non-maximal orders, so we would not need to assume $\omega(N)$ is
odd (when $F=\Q$) or $N$ is squarefree.  However, our dimension formulas from 
\cite{me:dim} are only for squarefree level because the trace formula we used is significantly 
more complicated for non-squarefree level, though the method should apply to
arbitrary level.  Potentially, this could make the hypotheses for a non-squarefree
analogue of \cref{thm1} considerably more complicated.

\medskip
\noindent
{\bf Acknowledgements.}
I thank Satoshi Wakatsuki for some discussions about mass formulas
which led me to think about the notion of $S$-ideal classes for quaternion algebras. 
Thanks are also due to Dan Fretwell, Catherine Hsu, Bao Viet Le Hung, and Chao Li for 
helpful discussions, as well as to the referee for useful comments.
I am grateful to the Simons Foundation for partial support through a 
Collaboration Grant.

%
%

\section{Quaternionic $S$-ideal classes} \label{sec:ideals}

Let $F$ be a totally real number field with narrow class number $h_F^+=1$,
and $B/F$ be a totally definite quaternion
algebra of discriminant $\frakN$.
Fix a maximal order $\calO$ of $B$.  For any (finite) prime $\frakp$ of $F$, we have
the local completions $B_\frakp = B \otimes_{F} F_\frakp$ and 
$\calO_\frakp = \calO \otimes_{\Z} \frako_{F,\frakp}$.  Then $B_\frakp/F_\frakp$
is a division algebra if and only if $\frakp | \frakN$. 
Let $\hat \calO^\times = \prod_\frakp \calO_\frakp^\times$
and $\hat B^\times = \prod^{\prime}_\frakp B_\frakp^\times$ denote the finite ideles 
of $B$, i.e., the
restricted direct product of the $B_\frakp^\times$'s with respect to the 
$\calO_\frakp^\times$'s.

When we restrict to $F=\Q$, we write $N$ for $\frakN$, $p$ for $\frakp$, and so on.

Recall there is a canonical bijection of
\begin{equation} \label{eq:clO}
 \Cl(\calO) := B^\times \bs \hat B^\times / \hat \calO^\times
\end{equation}
with the set (not a group) of right (locally principal) ideal classes of $\calO$.   
The class number $h_B = |\Cl(\calO)|$ is independent of the choice of $\calO$.

The number of maximal orders in $B$ up to $B^\times$-conjugacy is called the type number
$t_B$ of $B$.
The conjugacy classes of maximal orders are in bijection with
\begin{equation} \label{eq:type}
 B^\times \bs \hat B^\times / \hat G(\calO),
\end{equation}
where $\hat G(\calO) = \prod^{\prime} G(\calO_\frakp)$ is the stabilizer subgroup
with local components $G(\calO_\frakp) = \{ x \in B_\frakp^\times :
x \calO_\frakp x^{-1} = \calO_\frakp \}$.  Here $G(\calO_\frakp) = F_\frakp^\times \calO_\frakp^\times$ if
$\frakp \nmid \frakN$ and $G(\calO_\frakp) = B_\frakp^\times$ if $\frakp | \frakN$.  
The latter part follows
as any finite-dimensional $\frakp$-adic division algebra has a unique maximal order.  
Hence $\hat G(\calO) = \hat F^\times \hat \calO^\times \cdot \prod_{\frakp | \frakN} B_\frakp^\times$.
Since $t_B$ is the cardinality
of \eqref{eq:type} and $[B_\frakp^\times : F_\frakp^\times \calO_\frakp^\times] = 2$ at 
ramified places (and $h_F = 1$), one deduces that $\frac{h_B}{2^{\omega(\frakN)}} \le t_B \le h_B$, where
$\omega(\frakN)$ is the number of prime ideals dividing $\frakN$.

Let $S$ be a set of primes dividing $\frakN$.  We define the (right) \emph{$S$-ideal classes} of $\calO$ to be
\[ \Cl_S(\calO) := B^\times \bs \hat B^\times / \hat \calO_S^\times, \]
where
\[ \calO_S^\times = \prod_{\frakp \in S} B_\frakp^\times \times \prod_{\frakp \not \in S} \calO_\frakp^\times. \]
This interpolates \eqref{eq:clO} and \eqref{eq:type}, and is analogous to 
the definition of the $S$-ideal class group for number fields:
if $S= \emptyset$ one gets \eqref{eq:clO}, and if $S = \{ \frakp : \frakp | \frakN \}$  one
gets \eqref{eq:type}.  (The factor $\A_F^\times$ in the quotient \eqref{eq:type}
makes no difference since $h_F = 1$.)   The set $\Cl_S(\calO)$ is always finite.
Denote the $S$-ideal class number $|\Cl_S(\calO)|$ by $h_{B,S}$.
If $\frakM = \prod_{\frakp \in S } \frakp$, we sometimes also write
$\Cl_S(\calO) =: \Cl_\frakM(\calO)$ and $h_{B,S} =: h_{B, \frakM}$.

\section{Quaternionic modular forms} \label{sec:qmfs}

Let $F$, $B$, and $\calO$ be as above.  Let $\bfk = (k_1, \ldots, k_d) \in (2\mathbb \Z_{\ge 0})^d$,
where $d = [F:\Q]$.  Let $\tau_1, \ldots, \tau_d$ denote the embeddings of $F$ into
$\R$, and put  $B_\infty^\times = \prod B_{\tau_i}^\times$.  View each $B_{\tau_i}^\times$
as a subgroup of $\GL_2(\C)$.  Let $(\rho_{k_i}, V_{k_i})$ be the twist
$\det^{-k_i/2} \otimes \mathrm{Sym}^{k_i}$ of the $k_i$-th symmetric
power representation $\mathrm{Sym}^{k_i}$ of $\GL_2(\C)$ into $\GL_{k_i+1}(\C)$ restricted to 
$B_{\tau_i}^\times$.  The twist here gives $\rho_{k_i}$ trivial central character.
Put $(\rho_\bfk, V_\bfk) = \bigotimes (\rho_{k_i}, V_{k_i})$.

We define the space $\calM_\bfk(\calO)$ of weight $\bfk$ quaternionic modular forms of level $\calO$ to be the space of functions
$\phi : \hat B^\times \times B^\times_\infty \to V_\bfk$
satisfying
\[ \phi(\gamma xu, \gamma yg) = \rho_\bfk(g^{-1}) \phi(x,y) \quad \text{ for } x \in \hat B^\times, y \in B^\times_\infty,  \gamma \in B^\times, u \in \hat \calO^\times, \, g \in B^\times_\infty. \]
Alternatively, $\calM_\bfk(\calO)$ is the space of functions on
$B^\times \bs B^\times(\A) / \hat \calO^\times$ on which $B_\infty^\times$ acts
on the right by $\rho_\bfk$.
We note that a consequence of our assumption $h_F^+ = 1$ is that all forms
in $\calM_\bfk(\calO)$ are invariant under translation by the center $\A_F^\times$
of $B^\times(\A)$.  Without this assumption, 
we could restrict to the subspace of
forms with trivial central character as in \cite{me:cong}.

For the invariance conditions on $\phi$ to be compatible with the transformation condition on $B_\infty^\times$, it is necessary and sufficient that
$\phi(x,1) \in V_\bfk^{\Gamma(x)}$, where $\Gamma(x) = x \hat \calO^\times x^{-1}
\cap B^\times$.
Write 
\[ \Cl(\calO) = \{ x_1, \ldots, x_h \} \]
for some fixed choice of $x_1, \ldots, x_h$ in $B^\times(\A)$.
Put $\Gamma_i = \Gamma(x_i)$.
Then we can and will view the elements $\phi \in \calM_\bfk(\calO)$ as precisely the set of functions
\begin{equation} \label{eq:qmf-def2}
 \phi : \Cl(\calO) \to \bigsqcup V_\bfk^{\Gamma_i}, \quad \phi(x_i) \in V_\bfk^{\Gamma_i} 
\text{ for } 1 \le i \le h.
\end{equation}
Namely, we can view $\phi$ as a function of $\hat B^\times$ by $\phi(x) := \phi(x,1)$.
Since $\Cl(\calO)$ is precisely the set of orbits of 
$B^\times \bs B^\times(\A) / \hat \calO^\times$ under $B^\times_\infty$, any
$\phi \in \calM_\bfk(\calO)$ is completely determined by its values on $x_1, \ldots, x_h$.  
Consequently $\calM_\bfk(\calO) \simeq \bigoplus V_\bfk^{\Gamma_i}$.

Note that viewing $\phi$ as a function of $\hat B^\times$ (which we do from now on
except where explicated), $\phi$ is invariant under $\hat F^\times = Z(\hat B^\times)$,
right $\hat \calO^\times$-invariant and transforms on the left by $\rho_\bfk$ under $B^\times$
since
\begin{equation} \label{eq:leftBact}
\phi(\gamma x) = \phi(\gamma x,1) = \phi(x, \gamma^{-1}) = \rho_{\bfk}(\gamma) \phi(x, 1)
= \rho_\bfk(\gamma) \phi(x), \quad \gamma \in B^\times.
\end{equation}

If $\bfk = \zero := (0, 0, \ldots, 0)$, then $\rho_\bfk$ is the trivial representation 
so $\calM_\bfk(\calO)$ is simply the set of functions $\phi : \Cl(\calO) \to \C$.
Here we define the Eisenstein subspace $\calE_\zero(\calO)$ to be the space of 
$\phi \in \calM_\zero(\calO)$ which factors through the reduced norm $N_{B/F}$.
By the assumption that $h_F^+ = 1$, $\calE_\zero(\calO) = \C \one$, where
$\one$ denotes the constant function on $\Cl(\calO)$.  (For general $F$,
$\calE_\zero(\calO)$ is $h_F^+$-dimensional.)

In this case, we can define a normalized inner product on  $\calM_{\zero}(\calO)$ to be
\begin{equation} \label{eq:IP}
 (\phi, \phi') = \sum \frac {|\frako_F^\times|}{|\Gamma_i|} \phi(x_i) \overline{\phi'(x_i)}.
\end{equation}
Then we define the cuspidal subspace
$\calS_\zero(\calO)$ of $\calM_\zero(\calO)$ to be the orthogonal complement of the
Eisenstein subspace: $\calM_\zero(\calO) = \calE_\zero(\calO) \oplus \calS_\zero(\calO)$.

If $\bfk \ne \zero$, then nothing nonzero in $\calM_\bfk(\calO)$ can factor through
$\A^\times_F$, so we put $\calE_\bfk(\calO) = 0$ and 
$\calS_\bfk(\calO) = \calM_\bfk(\calO)$.

%
%

\subsection{Hecke operators}  \label{sec:hecke}
In this section, $g \in \hat B^\times$ and we view elements of $\calM_\bfk(\calO)$ as
functions on $\hat B^\times$ by \eqref{eq:qmf-def2}.

Fix a Haar measure $dg$ on $\hat B^\times$ which gives $\hat \calO^\times$
volume 1.  
For $\alpha \in \hat B^\times$, we
associate the Hecke operator $T_\alpha : \calM_\bfk(\calO) \to \calM_\bfk(\calO)$ given by
\begin{equation} \label{eq:hecke-def}
 (T_\alpha \phi)(x) = \int_{\hat \calO^\times \alpha
\hat \calO^\times} \phi(xg) \, dg.
\end{equation}

Writing $\hat \calO^\times \alpha \hat \calO^\times = \bigsqcup \beta_j \hat \calO^\times$
for some finite collection of $\beta_j \in \hat \calO^\times$, we can rewrite 
\eqref{eq:hecke-def} as the finite sum
\[ (T_\alpha \phi)(x) = \sum \phi(x \beta_j). \]

For $\frakp$ a prime of $F$, let $\varpi_\frakp$ denote a uniformizer in $F_\frakp$.
Then for $\frakp \nmid \frakN$, identify $B_\frakp^\times$ with $\GL_2(F_\frakp)$ and set 
$\alpha_\frakp = \bmx \varpi_\frakp & 0 \\ 0 & 1 \emx \in B_\frakp^\times$.  
For $\frakp | \frakN$, let $E_\frakp$ be the unramified quadratic extension of $F_\frakp$, 
write $B_\frakp = \{ \bmx x & \varpi_\frakp y \\ \bar x & \bar y \emx : x, y \in E_\frakp \}$ and
set $\alpha_\frakp = \varpi_{B_\frakp} = \bmx 0 & \varpi_\frakp \\ 1 & 0 \emx \in B_p^\times$.  Here we used
the notation $\varpi_{B_\frakp}$ to indicate that this element is a uniformizer for $B_\frakp$.

For any prime $\frakp$, let $T_\frakp = T_{\alpha_\frakp}$, where we view 
$\alpha_\frakp \in B_\frakp^\times$ as the
element $\beta = (\beta_v)_v \in \hat B^\times$ satisfying 
$\beta_v = \alpha_\frakp$ when $v=\frakp$ and $\beta_v = 1$ otherwise.
When $\frakp \nmid \frakN$, this definition agrees with the (suitably normalized)
definition of unramified Hecke operators for holomorphic Hilbert modular forms.

Suppose $\frakp | \frakN$.  Since $\calO_\frakp$ is the unique maximal order of $B_\frakp$, 
it is fixed under conjugation by $\alpha_\frakp = \varpi_{B_\frakp}$.  (In fact explicit calculation shows 
that conjugation by $\alpha_\frakp$ in $B_\frakp$ acts as the canonical involution of $B_\frakp$.) 
Consequently $\hat \calO^\times \alpha_\frakp \hat \calO^\times =
\varpi_{B_\frakp} \hat \calO^\times$, and the definition of the Hecke operator means
\begin{equation} \label{eq:ram-Tp}
 (T_\frakp \phi)(x) = \phi(x \varpi_{B_\frakp}), \quad \frakp | \frakN.
\end{equation}
Hence, for ramified primes, since $\varpi_{B_\frakp}^2 = \varpi_\frakp \in Z(\hat B^\times)$, we have $(T_\frakp^2 \phi)(x) = \phi(x \varpi_\frakp) = \phi(x)$, i.e.,
$T_\frakp$ acts on $\calM_\bfk(\calO)$ with order 2.

In this paper, we say $\phi \in \calM_\bfk(\calO)$ is an eigenform if it is an eigenfunction of 
\emph{all} $T_\frakp$'s.  Then $\calM_\bfk(\calO)$ has a basis of eigenforms as $(T_\frakp)_\frakp$ 
is a commuting family of diagonalizable operators.  Recall that this is not quite true for Hilbert (or 
elliptic) modular forms---rather, one either needs to restrict the definition of eigenforms to be 
eigenfunctions of the unramified Hecke operators or restrict to a subspace of newforms.  
In our quaternionic
situation, all eigenforms are ``new'' because we are working with a maximal order.  The 
diagonalizability of the ramified Hecke operators $T_\frakp$, $\frakp | \frakN$, 
follows from the fact that they are involutions.

Any eigenform $\phi \in \calM_\bfk(\calO)$ lies in an irreducible cuspidal automorphic
representation $\pi$ of $B^\times(\A)$ with trivial central character.  (Our definition of
cusp forms does not exactly match up with the usual notion of cuspidal automorphic
representations---the eigenform $\one \in \calM_\zero(\calO)$ is not a cusp form, and it
generates the trivial automorphic representation, which is a cuspidal representation of
$B^\times(\A)$ using the standard definition.  However, it will not correspond to a cuspidal
representation of $\GL_2(\A)$, which is why we do not call the form $\one$ a cusp form.)
At a ramified prime 
$\frakp$, the local representation $\pi_\frakp$ is 1-dimensional and factors through the 
reduced norm map $N_{B_\frakp/F_\frakp}$.  
Because we are working with trivial central character, either
$\pi_\frakp$ is the trivial representation $\one_\frakp$ or the reduced norm map composed with the unramified quadratic character $\eta_\frakp$ of $F_\frakp^\times$.  
Since $T_\frakp \phi = \pi(\varpi_{B_\frakp}) \phi$,
we see that $T_\frakp$ acts on $\phi$ by $+1$ (resp.\ $-1$) if $\pi_\frakp = \one_\frakp$
(resp.\ $\eta_\frakp \circ N_{B_\frakp/F_\frakp}$).

\subsection{The Jacquet--Langlands correspondence} \label{sec:JL}

The Jacquet--Langlands correspondence, proved in the setting of automorphic 
representations, gives an isomorphism:
\[ \calS_\bfk(\calO) \simeq S_{\bfk + \two}^\new(\frakN), \]
where $\two := (2, \ldots, 2) \in \mathbb N^d$.
This isomorphism respects the action of $T_\frakp$ for $\frakp \nmid \frakN$, i.e.,
it is an isomorphism of modules for the unramified Hecke algebra.  
(To get the right normalization of Hecke operators, 
we take the convention of viewing the space of Hilbert modular forms $M_{\bfk}(\frakN)$
adelically and defining the Hecke operators analogously to \eqref{eq:hecke-def}.)

Let $\mathrm{St}_\frakp$ denote the Steinberg representation of $\GL_2(F_\frakp)$.
For $\frakp | \frakN$, the Atkin--Lehner operator $W_\frakp$ acts on an eigenform
$f \in \calS_{\bfk + \two}^\new(\frakN)$ with eigenvalue $-1$ (resp.\ $+1$)  
if the associated local representation $\pi_{f,\frakp}$ is $\mathrm{St}_\frakp$ (resp.\ $\mathrm{St}_\frakp \otimes \eta_\frakp$).  In fact, we can take this to be the
definition of the Atkin--Lehner operator on the space of Hilbert modular newforms of 
squarefree level. (See \cite{shemanske-walling} for a more classical approach to Atkin--Lehner 
operators for Hilbert modular forms.) A standard computation shows that
the (normalized) ramified Hecke eigenvalue $a_\frakp(f) = - w_\frakp(f)$, i.e.,
$T_\frakp = -W_\frakp$ for $\frakp | \frakN$.

Since the local Jacquet--Langlands correspondence 
associates $\one_\frakp$ with
$\mathrm{St}_\frakp$ and $\eta_\frakp \circ N_{B_\frakp/F_\frakp}$ with
$\mathrm{St}_\frakp \otimes \eta_\frakp$, we see that the action of the ramified
Hecke operators $T_\frakp$ on $\calS_\bfk(\calO)$ corresponds to the
action of $T_\frakp = -W_\frakp$ on $\calS_{\bfk + \two}^\new(\frakN)$ under the 
Jacquet--Langlands correspondence.  This can be viewed as a 
representation-theoretic generalization of the relationship between the Fricke
involution on the space of weight 2 elliptic cusp forms and quaternionic theta
series given by Pizer \cite{pizer}.

While the Jacquet--Langlands correspondence is technically only a correspondence of cusp forms
(or rather, cuspidal representations which are not 1-dimensional), we can extend the above 
Hecke module isomorphism to include all of $\calM_\bfk(\calO)$.  

Namely, it suffices to assume $\bfk=\zero$, so $\calM_\bfk(\calO)$
is just the space of $\C$-valued functions on $\Cl(\calO)$.  
Then $\calE_\zero(\calO) = \C \one$, and
the $\frakp$-th eigenvalue of $\one \in \calE_\zero(\calO)$ 
is simply the degree of $T_\frakp$, i.e., $1+N(\frakp)$ if $\frakp \nmid \frakN$
or 1 if $\frakp | \frakN$.   There is an Eisenstein series $E_{\two, \frakN} \in M_\two(\frakN)$
with these same Hecke eigenvalues for all $\frakp$.  When $F=\Q$, we may take
$E_{2,N} := \sum_{d|N} \mu(d) d E_2(dz)$ where $E_2$ is the quasimodular weight 2 
Eisenstein series for $\mathrm{SL}_2(\Z)$ and $\mu$ is the M\"obius function. 
Thus when $\bfk = \zero$, we can extend
the above Hecke module isomorphism of cuspidal spaces to a Hecke module isomorphism:
\[ \calM_\zero(\calO) \simeq \C E_{\two, \frakN} \oplus
\calS_{\two}^\new(\frakN). \]
We take $w_\frakp(E_{\two, \frakN}) = -a_\frakp(E_{\two, \frakN}) = -1$ for all $\frakp | \frakN$.

We remark that for general $h_F^+$, the reduced norm map from 
$B$ to $F$ induces a surjective map $N_{B/F} : \Cl(\calO) \to \Cl^+(\frako_F)$,
and a basis of eigenforms for $\calE_\zero(\calO)$ is just the collection of maps 
$\lambda \circ N_{B/F}$ where $\lambda$ ranges over characters of $\Cl^+(\frako_F)$.
We can still extend the Jacquet--Langlands correspondence to all of 
$\calM_\zero(\calO)$ by associating $\lambda \circ N_{B/F}$ to
$E_{\two, \frakN} \otimes \lambda$.

\subsection{Relation with quaternionic $S$-ideal classes} \label{sec:33}

Let $\frakM$ be an integral ideal dividing $\frakN$, which we just
write as $M$ when $F=\Q$.
By a sign pattern $\chi = \chi_\frakM$ for $\frakM$, we mean a collection of
signs $\chi_\frakp \in \{ \pm 1 \}$ for all prime ideals $\frakp | \frakM$.
If $\chi_\frakp = +1$ (resp.\ $-1$) for all $\frakp | \frakM$, we denote the sign pattern by
$+_\frakM$ (resp.\ $-_\frakM$).  Also, if $\chi$ is a sign pattern for $\frakM$,
denote by $-\chi$ the sign pattern given by signs $-\chi_\frakp$ for all 
$\frakp | \frakM$.

Consider the subspace of $\calM_\bfk(\calO)$ with this
collection of Hecke signs:
\[ \calM_\bfk^\chi(\calO) = \langle \phi \in \calM_\bfk(\calO) \text{ is an eigenform} : T_\frakp \phi = \chi_\frakp \phi
\text{ for all } \frakp | \frakM \rangle. \]
Similarly we define $\calS_\bfk^\chi(\calO) = \calM_\bfk^\chi(\calO) \cap 
\calS_\bfk(\calO)$.  
Note that
$\calM_\bfk^\chi(\calO) = \calS_\bfk^\chi(\calO) \oplus \C \one$ if $\bfk = \zero$ and
$\chi = +_\frakM$; otherwise $\calM_\bfk^\chi(\calO) = \calS_\bfk^\chi(\calO)$.  

To keep notation consistent with \cite{me:dim} when $F=\Q$, we denote
the space of Hilbert newforms with fixed Atkin--Lehner (rather than Hecke)
signs by
\[ S_\bfk^{\new,\eps}(\frakN) = \langle f \in S_\bfk(\frakN) \text{ is a newform} : W_\frakp f = \eps_\frakp f \text{ for all } \frakp | \frakM \rangle, \]
for a sign pattern $\eps$ for $\frakM$.  The description of the Jacquet--Langlands
correspondence above tells us we have Hecke module isomorphisms:
\begin{equation} \label{eq:JLdim}
\calS_\bfk^\chi(\calO) \simeq S_{\bfk+\two}^{\new, -\chi}(\frakN),
\end{equation}
and
\begin{equation} \label{eq:JLdim2}
\calM_\bfk^\chi(\calO) \simeq 
\begin{cases}
\C E_{2,\frakN} \oplus 
S_{\two}^{\new, -\chi}(\frakN) & \text{if } \bfk = \zero \text{ and } \chi = +_\frakM,
\\
S_{\bfk+\two}^{\new, -\chi}(\frakN) & \text{else}.
\end{cases}
\end{equation}

If $\phi \in \calM_\bfk^{\chi}(\calO)$, then it is 
right $B_\frakp^\times$-invariant (i.e., $\phi(x \alpha_\frakp) = \phi(x)$ for all
$\alpha_\frakp \in B_\frakp^\times$) if and only if $\chi_\frakp = +1$.  This implies 
we can view forms in $\calM_\bfk^{+_\frakM}(\calO)$ as certain functions on $\Cl_S(\calO)$.  In particular, for weight zero we see that
\begin{equation} \label{eq:M0+}
\calM_\zero^{+_\frakM}(\calO) \simeq \{ \phi : \Cl_\frakM(\calO) \to \C \}.
\end{equation}
Hence
\begin{equation} \label{eq:hBS-dim}
 h_{B,\frakM} = \dim \calM_\zero^{+_\frakM}(\calO) = 1 + \dim S_\two^{\new, -_\frakM}(\frakN).
\end{equation}

We remark that when $F=\Q$ and $N=p$, we
have $h_{B,p} = t_B$ so \eqref{eq:hBS-dim} yields
$t_B = 1 + S_2^{\new,-_p}(p)$, which was already known to Deuring.  
More generally, but still with $F=\Q$, a relation between
type numbers and the full (not new) space of cusp forms with given Atkin--Lehner eigenvalues
was given by Hasegawa and Hashimoto \cite{hasegawa-hashimoto}, which is similar to, but slightly
different than, \eqref{eq:hBS-dim}.  Note they do not restrict to squarefree level, and 
their approach is essentially to use
explicit formulas for type numbers and dimensions, rather than looking through 
the lens of the Jacquet--Langlands correspondence as we do here.

When $F=\Q$, a formula for $\dim S_k^{\new, \eps}(N)$ was given in \cite{me:dim},
This translates into an explicit formula for the $S$-ideal class numbers $h_{B,S}$ by
\eqref{eq:hBS-dim}.  The general case is somewhat complicated, so here we just
explain the formula in a simple case which will arise for us later: when 
$S = \{ p \}$, we have $h_{B, p} = \frac 12 h_B = \frac 12 (1 + \dim S_2^\new(N))$ if 
(and only if) $p$ satisfies condition (a), (b), or (c) of \cref{prop:55} below.  

In the next section, we will generalize \eqref{eq:M0+}
to treat spaces $\calM^\chi_\bfk(\calO)$ of higher weight and other sign patterns $\chi$.

%
%

\section{Action of local involutions} \label{sec:action}

Keep the notation of the previous section.  Here, for a prime $\frakp$ at which
$B$ is ramified, we will study the action of $\varpi_{B_\frakp}$ on $\Cl(\calO)$.
This will give a ``local involution'' $\sigma_\frakp$ on the global space $\Cl(\calO)$,
which by \eqref{eq:ram-Tp} will tell us about the action of ramified Hecke operators on
$\calM_\bfk(\calO)$.    This will result in an algebro-combinatorial description of the
spaces $\calM_\bfk^\chi(\calO)$ for prescribed sign patterns $\chi$.

\subsection{Action on ideal classes} \label{sec41}
Let $\frakp$ be a prime at which $B$ ramifies.
For $S = \{ \frakp \}$, we also write $\Cl_S(\calO)$ as $\Cl_\frakp(\calO)$.
Now we have a surjective map
\begin{equation} \label{eq:cltoclp}
 \Cl(\calO) \to \Cl_\frakp(\calO)
\end{equation}
given by quotienting out on the right by $B_\frakp^\times$.   Since 
$B_\frakp^\times = F_\frakp^\times (\calO_\frakp^\times \sqcup \varpi_{B_\frakp} \calO_\frakp^\times)$,
given any 
$x \in \hat B^\times$ the associated $\{ \frakp \}$-ideal class 
$[x]_\frakp := B^\times x \hat \calO^\times B_\frakp^\times$
is either $[x]$ or $[x] \sqcup [x \varpi_{B_p}]$, where $[x] := B^\times x \hat \calO^\times$.
Thus the map \eqref{eq:cltoclp} has fibers of size 1 or 2.

Put another way, right multiplication by $\varpi_{B_\frakp}$ induces an involution, i.e.\ a 
permutation of order 2, on $\Cl(\calO)$, and the orbits of this involution are precisely the
fibers of \eqref{eq:cltoclp}.  Denote this involution by $\sigma_\frakp$, so 
$\sigma_\frakp([x]) = [x \varpi_{B_\frakp}]$ for any $x \in \hat B^\times$. 

It will be useful to know certain objects associated to ideal classes are
invariant under $\sigma_\frakp$.

For a right ideal $\calI$ of $\calO$, let 
$\calO_l(\calI) = \{ \alpha \in B : \alpha \calI \subset \calI \}$ denote the left order of $\calI$.
If $\calI$ corresponds to $x$, we also write the left order as
$\calO_l(x)$.  Note $x \hat \calO x^{-1} \cap B$ is a maximal order of $B$ since it 
locally is.  Since it preserves $x \hat \calO$ by left multiplication, we have 
$\calO_l(x) = x \hat \calO x^{-1} \cap B$.  From this it is easy to see that
$\calO_l(x) = \calO_l(x')$ for $x' \in [x]$, so we may unambiguously call this
the left order $\calO_l([x])$ of the ideal class $[x]$.
Similarly, since $\Gamma(x) = \calO_l(x)^\times$, this group only depends on
$[x]$ and we may also write it as $\Gamma([x])$.

\begin{lem}  \label{lem:41}
For $x \in \hat B^\times$, $\calO_l([x]) = \calO_l(\sigma_\frakp([x]))$
and $\Gamma([x]) = \Gamma(\sigma_\frakp([x]))$.
\end{lem}

\begin{proof} 
It suffices to prove the statement about left orders.
By the above adelic description of left orders,
 it suffices to show $\hat \calO^\times = \varpi_{B_\frakp}
\hat \calO^\times \varpi_{B_\frakp}^{-1}$.  Clearly these groups are the same away 
from $\frakp$, and they are the same at $\frakp$ since $B_\frakp$ has a unique maximal 
order.
\end{proof}

In this subsection, we needed to distinguish between $x$, $[x]$ and $[x]_\frakp$ for $x \in
\hat B^\times$, but below this is less crucial so we will use $x_i$ for
an both element of $\Cl(\calO)$ and a representative in $\hat B^\times$ as in \cref{sec:qmfs}.

\subsection{Action on quaternionic modular forms}

Fix a set of representatives $x_1, \ldots, x_h$ for $\Cl(\calO)$ and let
$\frakp | \frakN$.  Then we may view $\sigma_\frakp$ as a permutation on 
$\{ x_1, \ldots, x_h \}$.
Writing $\sigma_\frakp(x_i) = \gamma x_i \varpi_{B_\frakp} u$ for some
$\gamma \in B^\times$, $u \in \hat \calO^\times$, then by \eqref{eq:leftBact} we see
\[ \phi(\sigma_\frakp(x_i)) = \phi(\gamma x_i \varpi_{B_\frakp}) =
\rho_\bfk(\gamma) \phi(x_i \varpi_{B_\frakp}). \]
Note that $\gamma^{-1} \in \Gamma^{\sigma_\frakp}(x_i) := 
x_i \varpi_{B_\frakp} \hat \calO^\times \sigma_\frakp(x_i)^{-1} \cap B^\times.$
Thus the ramified Hecke action in \eqref{eq:ram-Tp} can be rewritten as
\begin{equation} \label{eq:ram-Tp2}
 (T_\frakp \phi)(x_i) = \rho_\bfk(\gamma) \phi(\sigma_\frakp(x_i)), \quad \text{for some } \gamma \in \Gamma^{\sigma_\frakp}(x_i),\text{ for all } 1 \le i \le h.
\end{equation}
We remark that for any fixed $\gamma_0  \in \Gamma^{\sigma_\frakp}(x_i)$, we 
can write any $\gamma \in \Gamma^{\sigma_\frakp}(x_i)$ as $\gamma = \gamma_0 \gamma'$ where $\gamma' \in \Gamma(\sigma_\frakp(x_i))$.  Hence
if the equation in \eqref{eq:ram-Tp2} holds for a fixed $i$ and some $\gamma 
\in \Gamma^{\sigma_\frakp}(x_i)$, it holds for all such $\gamma$ for that $i$ by
\eqref{eq:qmf-def2}.

Now let $\chi$ be a sign pattern for some $\frakM | \frakN$, and let
$\gamma_{i, \frakp} \in \Gamma^{\sigma_\frakp}(x_i)$ for each $1 \le i \le h$,
$\frakp | \frakM$.  Then for $\phi \in \calM_\bfk(\calO)$, we see that $\phi \in 
\calM_\bfk^\chi(\calO)$ if and only if
\begin{equation} \label{eq:chi-cond1}
 \phi(x_i) = \chi_\frakp \rho_\bfk(\gamma_{i,\frakp}) \phi(\sigma_\frakp(x_i)), \quad
 \text{for } 1 \le i \le h, \, \frakp | \frakM.
\end{equation}
In the case $\bfk=\zero$ so $\rho_\bfk$ is trivial, \eqref{eq:chi-cond1} simply
becomes 
\begin{equation}\label{eq:chi-cond2}
\phi(x_i) = \chi_\frakp \phi(\sigma_\frakp(x_i)), \quad \text{for } 1 \le i \le h, \, \frakp |
\frakM.
\end{equation}

If $\sigma_\frakp(x_i) = x_i$, put $V_\bfk^{\Gamma_i, \chi_\frakp} =
\{ v \in V_\bfk^{\Gamma_i} : \rho_\bfk(\gamma_{i,\frakp})v = \chi_\frakp v \}$.
Note that in this case $\gamma_{i,\frakp}^2 \in Z(B^\times)$, so $\gamma_{i,\frakp}$
acts as an involution and we have $V_\bfk^{\Gamma_i} \simeq V_\bfk^{\Gamma_i, +_\frakp}
\oplus V_\bfk^{\Gamma_i, -_\frakp}$.
If $x_i$ is not fixed by $\sigma_\frakp$, put $V_\bfk^{\Gamma_i, \chi_\frakp} = V_\bfk^{\Gamma_i}$.

\begin{lem} \label{lem:wtk} Fix $\chi_\frakp$ a sign for some $\frakp | \frakN$.  
Order $x_1, \ldots, x_h$ so that $x_1, \dots, x_t$ is a set of representatives for
$\Cl_\frakp(\calO)$, where $t = h_{B, \frakp}$.
Then we have an
isomorphism
\[ \calM_\bfk^{\chi_\frakp}(\calO) \simeq \{ \phi : \Cl_\frakp(\calO) \to 
\bigsqcup V_\bfk^{\Gamma_i, \chi_\frakp} \, | \, \phi(x_i) \in V_\bfk^{\Gamma_i, \chi_\frakp}
\text{ for } 1 \le i \le t \}. \]
\end{lem}

\begin{proof}
Let $\phi$ be an element of the set on the right, which we 
temporarily denote by $A(\chi_\frakp)$.  Then we extend $\phi$ to
$\Cl(\calO)$ as follows: for $t < j \le h$, write $x_j = \sigma_{\frakp}(x_i)$ for some
$1 \le i \le t$, and put $\phi(x_j) = \chi_\frakp \rho_\bfk(\gamma_{j,\frakp}) \phi(x_i)$.
Note that $\phi(x_j) \in V_\bfk^{\Gamma_j}$ by \cref{lem:41}.
This defines an embedding of $A(\chi_\frakp)$ into $\calM_\bfk^{\chi_\frakp}(\calO)$.
We will show surjectivity by a dimension argument.  

For $1 \le i \le t$, let $A_i(\chi_\frakp)$ be the subspace of $A(\chi_\frakp)$ consisting of elements
$\phi$ such that $\phi(x_j) = 0$ if $i \ne j$, $1 \le j \le t$.  If $\sigma_\frakp$ fixes
$x_i$, then  $V_\bfk^{\Gamma_i} \simeq V_\bfk^{\Gamma_i, +_\frakp}
\oplus V_\bfk^{\Gamma_i, -_\frakp}$ implies $\dim A_i(+_\frakp) + \dim A_i(-_\frakp)
= \dim V_\bfk^{\Gamma_i}$.  Otherwise $\sigma_\frakp(x_i) = x_j$ for some $j > t$,
and $\dim A_i(+_\frakp) = \dim A_i(-_\frakp) = \dim V_\bfk^{\Gamma_i} = \dim V_\bfk^{\Gamma_j}$.  Hence
\[ \dim A(+_\frakp) + \dim A(-_\frakp) = \sum_{i=1}^h \dim V_\bfk^{\Gamma_i}
= \dim \calM_\bfk(\calO), \]
and thus our embedding of $A(\chi_\frakp)$ into $\calM_\bfk^{\chi_\frakp}(\calO)$
must be surjective.
\end{proof}

There are two situations where the above description of $\calM_\bfk^{\chi_\frakp}(\calO)$ becomes simpler.  First, if $\sigma_\frakp$ has no fixed points, then we can identify this
space of forms with the functions $\phi$ on $\Cl_\frakp(\calO)$ such that 
$\phi(x_i) \in V_\bfk^{\Gamma_i}$ for each $1 \le i \le t$.  Second, if $\bfk = \zero$ then
we can identify this space with functions $\phi : \Cl_\frakp(\calO) \to \C$ such that
$\phi(x_i) = 0$ if $\sigma_\frakp(x_i) = x_i$ and $\chi_\frakp = -1$.

\subsection{Actions without fixed points}

Let $s_\frakp$ denote the number of orbits of size 2 for $\sigma_\frakp$, so $h-2s_\frakp$
is the number of fixed points of $\sigma_\frakp$.
For $\phi \in \calM_\zero(\calO)$, note the equation
$T_\frakp \phi = \phi$ imposes $s_\frakp$ linear constraints on $\phi$: 
$\phi(x_i) = \phi(\sigma_\frakp(x_i))$
for $x_i$ in any orbit of size 2.  On the other hand, $T_\frakp \phi = -\phi$ forces
$\phi(x_i)=0$ for any $x_i$ fixed by $\sigma_\frakp$ and $\phi(x_i) = -\phi(\sigma_\frakp(x_i))$ for
$x_i$ in an orbit of size 2.  Hence for a sign pattern $\chi_\frakp$ for $\frakp$, we 
have
\begin{equation} \label{sp-dim}
\dim \calM_\zero^{\chi_\frakp}(\calO) = 
\begin{cases}
h - s_\frakp & \chi_\frakp = +1 \\
s_\frakp & \chi_\frakp = -1.
\end{cases}
\end{equation}
Consequently, we can compute $s_\frakp$ from \eqref{eq:JLdim2} and
a dimension formula for $S_2^{\new, -\chi_\frakp}(\frakN)$.
In particular, $\sigma_\frakp$ acts without fixed points if and only if
\[ \dim S_2^{\new, +_\frakp}(\frakN) =\dim S_2^{\new, -_\frakp}(\frakN) + 1. \]

Now we assume $F=\Q$, and will use a trace formula for the Atkin--Lehner operator
$W_p$ on $S_2^\new(N)$ from \cite{me:dim}
to give necessary and sufficient criteria for $\sigma_p$ to act on $\Cl(\calO)$ without
fixed points, which is equivalent to $s_p = \frac h2$.  

\begin{lem} \label{lem:fixpt}
Let $p | N$.  

(a) For $p > 2$, $s_p = \frac h2$ if and only if ${-p \leg q} = 1$ for some odd prime
$q | N$ or if $N$ is even and $p \equiv 7 \mod 8$.

(b) For $p = 2$, $s_p = \frac h2$ if and only if $N$ is divisible by a prime which is $1 \mod 4$
and ${-2 \leg q} = 1$  for some prime $q | N$.
\end{lem}

\begin{proof} 
By \eqref{sp-dim}, $s_p = \frac h2$ if and only if
$\dim S_2^{\new, +_p} = 1 + \dim S_2^{\new, -_p}$, i.e., if and only if
$\tr_{S_2^\new(N)} W_p = 1$.  This trace is computed in \cite[Prop 1.4]{me:dim}.

Let $N' = N/p$.  For $m \in \mathbb N$, let $m_\odd=  2^{-v_2(m)}m$ be the odd part of $m$.  
We define a constant $b(p,N')$ by the following table:
\begin{center}
\begin{tabular}{c|c|c}
& $b(p, N')$ & $b(p, N')$ \\
$p \mod 8$ &  for $N'$ odd &  for $N'$ even \\
\hline
1, 2, 5, 6 & 1 & $-1$\\
3 & 4 & $-2$ \\
7 & 2 & 0
\end{tabular}
\end{center}

If $p > 3$, the trace of interest is
\begin{equation}
\tr_{S_2^\new(N)} W_p = 1 - \frac 12 |\Cl(\Q(\sqrt{-p}))| b(p, N') \prod_{q | N'_\odd} \left( {-p \leg q} - 1 \right).
\end{equation}
This is 1 if and only if the second term on the right is 0, which gives part (a) when $p > 3$.
If $p=3$, this trace is
\begin{equation}
\tr_{S_2^\new(N)} W_3 = 1 -  (-1)^{v_2(N')} \prod_{q | N'_\odd} \left( {-3 \leg q} - 1 \right).
\end{equation}
This finishes (a).

If $p=2$, this trace is
\begin{equation}
\tr_{S_2^\new(N)} W_2 =1 - \frac 12 \left( \prod_{q | N'} \left( {-2 \leg q} - 1 \right) + \prod_{q | N'} \left( {-1 \leg q} - 1 \right) \right).
\end{equation}
This gives (b).
\end{proof}

We remark that knowing the traces of the Atkin--Lehner operator $W_p$ on $S_2^\new(N)$
is the same as knowing the $S$-ideal class numbers $h_{B, p}$ together with $h$ (see
\eqref{eq:hBS-dim} and \cite{me:dim}), so one may view the above result
as an application of formulas for
$S$-ideal class numbers, i.e., an application of the refined dimension formulas for
$S_2^{\new,\eps}(N)$.

\subsection{Weight zero spaces}
To study the spaces $\calM_\bfk^\chi(\calO)$ in more detail, 
we need to understand how the involutions $\sigma_\frakp$ interact for
the various primes $\frakp | \frakM$.  It will be convenient to describe this in terms of a graph.
The general case is somewhat complicated, so here we treat weight zero before discussing
higher weights.

Fix an integral ideal $\frakM | \frakN$ and  a sign pattern $\chi$
for $\frakM$.  We associate to $\chi$ a (signed multi)graph $\Sigma_\chi$ as follows.
Let the vertex set of $\Sigma_\chi$ be $\Cl(\calO) = \{ x_1, \ldots, x_h \}$.  
For $\frakp | \frakM$,
let $E(\chi_\frakp)$ denote the set of signed edges 
$\{ \chi_\frakp \cdot (x_i, \sigma_\frakp(x_i)) \}$ where $x_i$ runs over a complete set of
representatives for the orbits of $\sigma_\frakp$.  (By signed edges, we mean weighted
edges, where the weights are $\pm 1$ according to whether $\chi_\frakp = \pm 1$.)
Then we let the edge set of $\Sigma_\chi$ be the disjoint union of the $E(\chi_\frakp)$'s.
In other words, to construct our graph $\Sigma_\chi$ on $\Cl(\calO)$,
for all $1\le i \le j \le h$ and $\frakp | \frakM$, we add an (undirected) edge 
between $x_i$ and $x_j$ with sign $\chi_\frakp$ if and only if
$x_j = \sigma_\frakp(x_i)$.  Note that $\Sigma_\chi$ may
have loops as well as multiple edges with the same or opposite signs.

Let $X_1, \ldots, X_t$ denote the (vertex sets of the) connected components of $\Sigma_\chi$.  
We note that $X_1, \ldots, X_t$ do not
depend upon $\chi$---the sign pattern only affects the signs of the edges in
$\Sigma_\chi$.  Moreover, $x_j$ lies in the connected component of $x_i$ if and only
if it lies in the orbit of $x_i$ under the permutation group generated by $\{ \sigma_\frakp :
\frakp | \frakM \}$.  By the description of $\sigma_\frakp$ in terms of \eqref{eq:cltoclp},
this is equivalent to $x_j$ lying in the same $S$-ideal class as $x_i$, where
$S = \{ \frakp : \frakp | \frakM \}$.  Hence, viewing the $S$-ideal classes as subsets
of $\Cl(\calO)$, we may write $\Cl_S(\calO) = \{ X_1, \ldots, X_t \}$, and we see 
$t = h_{B,S}$.

Let $E_i$ be the edge set for $X_i$ in $\Sigma_\chi$ and 
partition $E_i = E_i^+ \sqcup E_i^-$, where 
$E_i^{\pm}$ denotes the subset of edges with sign $\pm 1$.  We say $X_i$ 
is $\chi$-admissible if there is a partition $X_i = X_i^+ \sqcup 
X_i^-$ such that the set of edges in $E_i$ which connect a vertex in $X_i^+$ with a vertex in $X_i^-$ is precisely $E_i^-$.  In this case, we call the partition $X_i^+ \sqcup X_i^-$ $\chi$-admissible.  Note that if $\chi = +_\frakM$, then $X_i^+ = X_i$ and $X_i^- = \emptyset$ is
always a $\chi$-admissible partition of $X_i$.

Denote the set of $\chi$-admissible $X_i \in \Cl_S(\calO)$ by $\Cl_S(\calO)^{\chi-\adm}$.

\begin{prop} \label{prop:wt0}
Let $\chi$ be a sign pattern for $\frakM | \frakN$, 
$S = \{ \frakp : \frakp | \frakM \}$, and write $\Cl_S(\calO) = \{ X_1, \ldots, X_t \}$.  
Then we have an isomorphism
\[ \calM_\zero^\chi(\calO) \simeq \{ \phi : \Cl_S(\calO)^{\chi-\adm} \to \C \}. \]
\end{prop}

Note that when $\chi = +_\frakM$, every class in $\Cl_S(\calO)$ is $\chi$-admissible
so this gives \eqref{eq:M0+}.

\begin{proof}
Order $x_1, \ldots, x_h$ so that $x_i \in X_i$ for $1 \le i \le t$.
Let $\phi \in \calM_\zero^\chi(\calO)$.  By \eqref{eq:chi-cond2}, if $x_{j_1}$ are $x_{j_2}$ are vertices in $X_i$ connected by an edge with sign $\pm 1$, then $\phi \in \calM_\zero^\chi(\calO)$ means $\phi(x_{j_1}) = \pm \phi(x_{j_2})$.  Hence
the value of $\phi(x_j)$
is determined by $\phi(x_i)$ (namely, is $\pm \phi(x_i)$) whenever $x_j \in X_i$.  
This gives a map from
$\calM_\zero^\chi(\calO)$ into the space of functions on $\Cl_S(\calO)^{\chi-\adm}$ by
restricting $\phi$ to be a function on the elements $x_i$, $1 \le i \le t$,
such that  $X_i$ is $\chi$-admissible.

To show this map is a bijection, it suffices to show that for $1 \le i \le t$
there exists $\phi \in \calM_\zero^\chi(\calO)$ such that $\phi(x_i) \ne 0$ if and only if
$X_i$ is $\chi$-admissible.  If $\phi \in \calM_\zero^\chi(\calO)$ with
$\phi(x_i) \ne 0$, then the partition of
$X_i$ into the two sets $X_i^+ = \{ x_j \in X_i : \phi(x_j) = \phi(x_i) \}$ and
$X_i^{-} = \{ x_j \in X_i : \phi(x_j) = -\phi(x_i) \}$ is a $\chi$-admissible partition of $X_i$.
Conversely, if $X_i^+ \sqcup X_i^-$ is a $\chi$-admissible partition of $X_i$, then
we can define an element of $\phi \in \calM_\zero^\chi(\calO)$ by setting
$\phi(x_j) = \pm 1$ if $x_j \in X_i^\pm$ and $\phi(x_j) = 0$ if $x_j \not \in X_i$.
\end{proof}

Thus $\dim \calM_\zero^\chi(\calO)$ is the number of $\chi$-admissible classes
in $\Cl_S(\calO)$, which generalizes \eqref{eq:hBS-dim}.  For congruences applications,
we want to know more about which $X_i$ are admissible.  Clearly we have

\begin{cor} All $X_i \in \Cl_S(\calO)$ are $\chi$-admissible if and only if 
$\dim \calM_\zero^\chi(\calO) = \dim \calM_\zero^{+_\frakM}(\calO)$.
\end{cor}

It does not seem easy to say exactly what $\Sigma_\chi$ looks like in general, however we can
get some information from considering how the edge sets $E(\chi_\frakp)$ can interact for
various $\frakp$.

\begin{lem} If $\frakM = \frakp \frakM_0$ and $X \in \Cl_{\frakM_0}(\calO)$,
then there exists $X' \in \Cl_{\frakM_0}(\calO)$ such that $x_i \in X$ implies
$\sigma_\frakp(x_i) \in X'$.
\end{lem}

\begin{proof}
The projection $\Cl_{\frakM_0}(\calO) \to \Cl_\frakM(\calO)$ has fibers of size 1 or 2.
If the fiber containing $X$ has size 1, the lemma is true with $X' = X$.  Otherwise,
let $X'$ be the other element in the fiber containing $X$.  Then there exists
$x_i \in X$ such that $\sigma_\frakp(x_i) \in X'$, i.e., $x_i \varpi_{B_\frakp} \in X'$.
One easily sees that this implies $x_j \varpi_{B_\frakp} \in X'$ for all 
$x_j \in X = B^\times x_i \hat \calO^\times \prod_{\frakq | \frakM_0} B_\frakq^\times$.
\end{proof}

Thus if we think of building $\Sigma_\chi$ in stages by adding the edge sets $E(\chi_\frakp)$ 
one prime at a time, we see that at each stage each connected component comprises exactly
one or two connected components from the previous stage.  Furthermore, if a connected
component is obtained by linking two connected components $X$ and $X'$ from the previous stage, then involution $\sigma_\frakp$ linking $X$ and $X'$ must be a bijection between
the set of right $\calO$-ideal classes in $X$ and those in $X'$.

Consequently, each connected component $X_i \in \Sigma_\chi$ has cardinality
$2^m$ for some $0 \le m \le 2^{\omega(\frakM)}$.

\subsection{Admissibility in higher weight}

Now we return to arbitrary weight $\bfk \in (2\Z_{\ge 0})^d$.

As before, let $\frakM | \frakN$ and put $S = \{ \frakp | \frakM \}$.  
Write $\Cl_S(\calO) = \{ X_1, \ldots, X_t \}$ and $\Cl(\calO) = \{ x_1, \ldots, x_h \}$
with $x_i \in X_i$ for $1 \le i \le t$.  For a sign pattern $\chi$ for $\frakM$, we say
$X_i$ is $\chi$-admissible in weight $\bfk$ if for any $v \in V_\bfk^{\Gamma_i}$ there
exists $\phi \in \calM_\bfk^\chi(\calO)$ such that $\phi(x_i) = v$.  By the proof of
\cref{prop:wt0}, being $\chi$-admissible in weight $\zero$ is just the notion of
$\chi$-admissible from the previous section.

If every $X_i$ is $\chi$-admissible in weight $\bfk$, then similar to previous sections to we get
an isomorphism
\begin{equation} \label{eq:Mk-ClS}
 \calM_\bfk^\chi(\calO) \simeq \{ \phi : \Cl_S(\calO) \to 
\bigsqcup V_\bfk^{\Gamma_i} \, | \, \phi(x_i) \in V_\bfk^{\Gamma_i}
\text{ for } 1 \le i \le t \}
\end{equation}
by simply restricting $\phi \in \calM_\bfk^\chi(\calO)$ to $x_1, \ldots, x_t$.
Without the admissibility condition, there is always an injection from the set on the
left to the set on the right, and we see that 
$\dim \calM_\bfk^\chi(\calO) = \sum_{i=1}^t \dim V_\bfk^{\Gamma_i}$ if
and only if each $X_i$ is $\chi$-admissible in weight $\bfk$.

\begin{lem} \label{lem:equidist}
Suppose $\dim \calM^\chi_\zero(\calO) = \dim \calM^{\chi'}_\zero(\calO)$ for any
choices of sign patterns $\chi, \chi'$ for $\frakM$.  Then for any $\bfk$, 
sign pattern $\chi$ for $\frakM$ and $X_i \in \Cl_S(\calO)$, we have 
that $X_i$ is $\chi$-admissible in weight $\bfk$.  Moreover, for fixed
$\bfk$, the spaces $\calM^\chi_\bfk(\calO)$ have the same dimension for all $\chi$.
\end{lem}

\begin{proof}
We prove this by induction on $\frakM$.  It is vacuously true for $\frakM=\frako_F$, so
suppose $\frakM = \frakp_0 \frakM_0$ and assume the lemma is true for $\frakM_0$.
Write $\Cl_S(\calO) = \{ X_1, \ldots, X_t \}$ and order $x_1, \ldots, x_h$ so $x_i \in X_i$
for $1 \le i \le t$.   Put $S_0 = \{ \frakp | \frakM_0 \}$.
The hypothesis in the lemma with $\chi, \chi'$ taken to be the two
sign patterns for $\frakM$ which restrict to $+_{\frakM_0}$ for $\frakM_0$ 
implies $\Cl_S(\calO) = \frac 12 \Cl_{S_0}(\calO)$ by \eqref{eq:hBS-dim}.
Then for any $X_i \in \Cl_S(\calO)$, we may write $X_i = Y_i \sqcup Y_i'$ where
$Y_i, Y_i' \in \Cl_{S_0}(\calO)$.  By \cref{lem:41}, the $\Gamma_j$'s
are the same for all $x_j \in X_i$.

Fix a sign pattern $\chi$ for $\frakM$ and let $\chi_0$ be the restriction of $\chi$ to $S_0$.  Let $\chi'$ be the extension of $\chi_0$ to $S$ such that $\chi'_{\frakp_0} = - \chi_{\frakp_0}$.
On one hand, we have
\[ \dim \calM_{\bfk}^{\chi_0}(\calO) = 2 \sum_{i=1}^t \dim V_\bfk^{\Gamma_i}. \]
On the other hand, we have
\[ \dim \calM_{\bfk}^{\chi_0}(\calO) = \dim \calM_{\bfk}^{\chi}(\calO)
+  \dim \calM_{\bfk}^{\chi'}(\calO). \]
But each of the dimensions on the right is at most
$\sum_{i=1}^t \dim V_\bfk^{\Gamma_i}$, so our previous equation means in fact
$ \dim \calM_{\bfk}^{\chi}(\calO) =  \dim \calM_{\bfk}^{\chi'}(\calO) =
\sum_{i=1}^t \dim V_\bfk^{\Gamma_i}$.  This implies both the admissibility and
dimension assertions.
\end{proof}

\begin{cor} \label{cor:admQ}
Suppose $F = \Q$ and $M | N$ such that, for each divisor $d | M$ with $d \ne 1$,
there exists an odd $p | \frac NM$ such that ${-d \leg p} = 1$.  If $M$ is even, we further assume $\frac NM$ is divisible by a prime $p \equiv 1 \mod 4$.  Then each $X_i \in \Cl_M(\calO)$ is $\chi$-admissible in weight $\bfk$ for all weights $\bfk$ and sign patterns
$\chi$ for $M$.
\end{cor}

\begin{proof}
By the lemma, we want to know that the sign patterns for $M$ are perfectly equidistributed
in the space $\calM_\zero(\calO)$, i.e.,  that $\dim S_2^{\new, \eps}(N) =
\dim S_2^{\new, -_M}(N) + 1$ for all sign patterns $\eps$ for $M$ with $\eps \ne -_M$.
This is immediate from \cite[Thm 3.3]{me:dim} (which also immediately implies
the sign patterns for $M$ are perfectly equidistributed in higher weight).
\end{proof}

\begin{rem} \label{rem:41}
If $M$ is prime which is $7 \mod 8$ and $N$ is even, then the conclusion of the
corollary also holds by \cref{lem:fixpt}.
\end{rem}

\section{Congruences}

Now we prove a congruence result under admissibility hypotheses.  In particular,
we will find that
equidistribution of sign patterns in weight 0 implies sign patterns are
in some sense equidistributed in congruence classes in all weights.

Let $F, B, \calO, \frakN$ be as above.  Fix a set of representatives $x_1, \ldots, x_h$
for $\Cl(\calO)$.

\subsection{Integrality}

First we describe some notions and properties of integrality.  

Recall $\tau_1, \ldots, \tau_d$ are the embeddings of $F$ into $\C$.  Let
$E/F$ be a totally imaginary quadratic extension which splits $B$.  
Then we may fix an embedding of $B$ into $M_2(E)$ so that $\calO$ maps into 
$M_2(\frako_E)$.
If $v_i$ is the place of $F$ associated to $\tau_i$, the embedding of $B$ into $M_2(E)$
induces an embedding $\tau^B_i: B_{v_i} \to M_2(\C)$ such that $\calO$ maps into
$M_2(\frako_{E_i})$, where $E_i$ the image of $E$ under an extension of $\tau_i$.
We take these embeddings in our definition of $(\rho_\bfk, V_\bfk)$.  In particular,
$\rho_{k_i}(\gamma) \in M_{k_i+1}(\frako_{E_i})$ for $\gamma \in \calO$.

Let $R \subset \C$ be a ring such that $\tau_i^B(\calO) \subset M_2(R)$
for all $1 \le i \le d$.   Realizing $V_\bfk = \C^n$,
let $V_\bfk(R) = R^n$ be the subspace of ``$R$-integral vectors.''  
We say $\phi \in \calM_\bfk(\calO)$ is $R$-integral (with respect 
to $x_1, \ldots, x_h$) if $\phi(x_i) \in V_\bfk(R)$ for $1 \le i \le h$.
Let $\calM_\bfk(\calO; R)$ be the $R$-submodule of $R$-integral forms in 
$\calM_\bfk(\calO)$ (with respect to $x_1, \ldots, x_h$).  

Recall for any Hecke operator $T = T_\alpha$, there exists a finite collection of 
$\beta_j \in \hat B^\times$ such that for any $\phi \in \calM_\bfk(\calO)$,
\[ (T\phi)(x) = \sum \phi(x \beta_j). \]
For any $1 \le i \le h$, we can write $x_i \beta_j = z_{ij} \gamma_{ij} x_{m_{ij}} u_{ij}$
for some $z_{ij} \in Z(B^\times)$, $\gamma_{ij} \in B^\times \cap \calO$,
$1 \le m_{ij} \le h$, and $u_{ij} \in \hat \calO^\times$.  
Then
\[ (T\phi)(x_i) = \sum \rho_\bfk(\gamma_{ij}) \phi(x_{m_{ij}}), \quad 1 \le i \le h. \]
By our integrality condition on $\gamma_{ij}$ and assumptions on $R$, $T\phi$ is $R$-integral when $\phi$ is.  

Moreover, viewing $\phi$ as a vector in $\C^{nh}$ formed by concatenating the vectors $\phi(x_i) \in \C^n$ for $1 \le i \le h$, 
we can think of $T$ as given by a $nh \times nh$ Brandt matrix
with entries in $R$ (in fact in $\Z_{\ge 0}$ when $\bfk = \zero$).  
Since there exists a Hecke operator $T=T_\alpha$
with distinct eigenvalues, a basis of 
eigenforms of $\calM(\calO)$ can be described as a complete set of eigenvectors 
for some $R$-integral matrix $T$.  Thus $\calM(\calO)$ has a basis consiting of 
$R$-integral eigenforms for some integer ring $R$.

For integral $\phi, \phi' \in \calM_{\bfk}(\calO; R)$ and an ideal $\ell$ of $R$, 
we write $\phi \equiv \phi' \mod \ell$
if the vectors $\phi(x_i)$ and $\phi'(x_i)$ are coordinate-wise congruent mod $\ell$
for all $i$.

\subsection{Congruences under admissibility}

Let $\frakM | \frakN$ and $S = \{ \frakp | \frakM \}$.

\begin{thm} \label{mainthm}
Let $\phi \in \calM_\bfk(\calO)$ be an eigenform, $\chi$
a sign pattern for $\frakM$, and $\ell | 2$ a prime of $\bar \Q$.  Suppose 
each $X \in \Cl_S(\calO)$ is $\chi$-admissible in weight $\bfk$.  Then
there exists an eigenform $\phi' \in \calM_\bfk^\chi(\calO)$ such that
$a_\frakp(\phi) \equiv a_\frakp(\phi') \mod \ell$ for all primes $\frakp$ of $F$.
\end{thm}

\begin{proof}
 Let $K$ be sufficiently large number field.  Namely, assume $K$ contains
the rationality fields for all eigenforms in $\calM_\bfk(\calO)$ and 
$\tau_i^B(\calO) \subset M_2(\frako_K)$ for all $1 \le i \le d$. 
We may assume $\phi$ is $\frako_K$-integral with respect to $x_1, \ldots, x_h$.

Let $\calI$ be a prime ideal of $\frako_K$ under $\ell$, and $R$ the localization of
$\frako_K$ at $\calI$.
A priori, if $\calI$ is not principal, it may not be possible to scale the values of $\phi$ so that
$\phi$ is $R$-integral and $\phi \not \equiv 0 \mod \calI$, but we can pass to a finite
extension of $K$ (i.e., enlarge $K$ if necessary) that principalizes $\calI$ to assume this.

Let $\eps$ be the sign pattern for $\frakM$ such that $\phi \in \calM_\bfk^\eps(\calO)$.
Write   $\Cl_S(\calO) = \{ X_1, \ldots, X_t \}$ 
and order $x_1, \ldots, x_h$ so that $x_i \in X_i$ for $1 \le i \le t$.
For each $\frakp | \frakM$ and $1 \le i \le h$, 
let $\gamma_{i,\frakp} \in \Gamma^{\sigma_\frakp}(x_i)$.
Then $\phi$ is determined by $\phi(x_1), \ldots, \phi(x_t)$ and 
$\phi(x_i) = \eps_\frakp \rho_\bfk(\gamma_{i,\frakp}) \phi(\sigma_\frakp(x_i))$
for all $i, \frakp$.

We define a function $\phi'$ on $\Cl(\calO)$ as follows.  For $1 \le i \le t$, let
$\phi'(x_i) = \phi(x_i)$.  Extend $\phi'$ to $\Cl(\calO)$ by requiring
$\phi(x_i) = \chi_\frakp \rho_\bfk(\gamma_{i,\frakp}) \phi(\sigma_\frakp(x_i))$
for all $i, \frakp$.  Then $\phi' \in \calM^\chi_\bfk(\calO)$ by \eqref{eq:Mk-ClS}, 
and $\phi'(x_i) = \pm \phi(x_i)$ for $1 \le i \le h$.
Thus $\phi' \equiv \phi \mod 2$ with respect to 
$x_1, \ldots, x_h$.  
However, this $\phi'$ need not be an eigenform.

Take a basis of $\calM_\bfk(\calO)$ consisting of eigenforms 
$\phi_1, \ldots, \phi_m \in \calM_\bfk(\calO; K)$ such that $\phi_1 = \lambda \phi$
for some $\lambda \in K^\times$.  Since
$\phi' \in \calM_\bfk(\calO; R)$ we have that $\phi' = \sum c_j \phi_j$ for some
$c_j \in K$.  By rescaling our basis vectors if necessary, 
we may assume
$\phi'$ and $\phi$ are $R$-linear combinations of $\phi_1, \ldots, \phi_m$.    

Let $M$ be the $R$-module generated by $\phi_1, \ldots, \phi_m$,
and $M^\chi$ be the submodule generated by the collection of $\phi_j$'s which lie
in $\calM^{\chi}_\bfk(\calO)$.  Then $\phi' \in M^\chi$.

Then the integrality property of
Hecke operators implies each Hecke operator $T_\alpha$ acts on $M$ as well
as $M/\calI M$.  Now $\phi' \equiv \phi \mod \calI$, so the image of 
$\phi'$ in $M^\chi/\calI M^\chi$ is a (nonzero) mod $\calI$ eigenvector of each $T_\alpha$.  
The Deligne--Serre lifting lemma \cite[Lem 6.11]{deligne-serre} now tells us there
is an eigenform $\phi'' \in M^\chi$, i.e.\ some $\phi_j$, 
which has the same mod $\calI$ Hecke eigenvalues  
as $\phi'$, and thus $\phi$.  (Note the Deligne--Serre lemma does not tell us that we may 
take $\phi'' \equiv \phi \mod \calI$---cf.\ \cite[(3.3)]{me:cong}.)
\end{proof}

Let $S_\bfk^\new(\frakN)^*$ be the space $S_\bfk^\new(\frakN)$ if $\bfk \ne \two$
and $S_\two^\new(\frakN) \oplus \C E_{2, \frakN}$ if $\bfk = \two$.  Similarly,
for a sign pattern $\eps$ for $\frakM$,
let $S_\bfk^{\new,\eps}(\frakN)^*$ be $S_\bfk^{\new,\eps}(\frakN)$ unless 
$\bfk=\two$ and $\eps = -_\frakM$, in which case it is 
$S_\two^{\new,-_\frakM}(\frakN) \oplus \C E_{2, \frakN}$.

\begin{cor} \label{cor1}
Let $\frakM | \frakN$.
Suppose all sign patterns $\eps$ for $\frakM$ are equidistributed in the space
$S_\two^\new(\frakN)^*$, i.e.\ $\dim S_\two^{\new,\eps}(\frakN)^*$ is
independent of $\eps$.  Let $\bfk \in (2\mathbb N)^d$, 
$f$ an eigenform in $S_\bfk^\new(\frakN)^*$ and
$\ell$ a prime of $\bar \Q$ above $2$.  Then for any
sign pattern $\eps$ for $\frakM$, there exists an eigenform 
$g \in S_\bfk^{\new,\eps}(\frakN)^*$ such that
$a_\frakp(f) \equiv a_\frakp(g) \mod \ell$ for all primes $\frakp$.
In particular, there are at least $2^{\omega(\frakM)}$ eigenforms in $S_\bfk^\new(\frakN)^*$
which are congruent to $f$ mod $\ell$.
\end{cor}

\begin{proof}
Use the Jacquet--Langlands correspondence, \cref{lem:equidist}, and the above theorem.
\end{proof}

By \cref{cor:admQ}, this gives \cref{thm1} when $F=\Q$ excepting the assertion that
we can take $g \in S_2^{\new, -_M}(N)$ when the weight $k=2$, $\eps = -_M$ and
$N$ is not an even product of three primes.  We handle this below.

Since, in weight $\zero$, all quaternionic $S$-ideal classes are $+_\frakM$-admissible, the
above theorem also gives the following.

\begin{cor} \label{cor2} 
Let $f \in S_\two(\frakN)$ be a newform and $\ell$ a prime of $\bar \Q$
above $2$.  Then there exists an eigenform
$g \in S_\two(\frakN) \oplus \C E_{2,\frakN}$ such that
$a_\frakp(f) \equiv a_\frakp(g) \mod \ell$ for all $\frakp$ and
$a_\frakp(g) = +1$ for all $\frakp | \frakN$.
\end{cor}

This gives \cref{thm2} when $F=\Q$ excepting the assertion about when we may take
$g$ cuspidal.

\subsection{Eisenstein and non-Eisenstein congruences} \label{sec:eis}
Here we will refine the latter corollary to show that when $F=\Q$, we can take 
$g \in S_2^{\new, -_M}(N)$ if $N$ is not an even product of three primes, which will
finish the proof of both \cref{thm1} and \cref{thm2}.

First we refine the main theorem of \cite{me:cong} in the setting
$\ell = 2$ and $h_F^+=1$.  (The proof is also similar.)

\begin{prop} \label{prop:54}
Suppose the numerator of $2^{1-d} |\zeta_F(-1)|N(\frakN)
\prod_{p | \frakN} (1- N(\frakp)^{-1})$ is even and the type number $t_B > 1$.  Then there exists an newform
$g \in S_{\two}(\frakN)$ such that $w_\frakp(g) = -1$ for all $\frakp | \frakN$
and $a_{\frakp}(g) \equiv a_{\frakp}(E_{2, \frakN}) \mod 2$ for all $\frakp$.
\end{prop}

\begin{proof} Consider the graph $\Sigma_\chi$ described above when 
$\chi$ is the sign pattern $+_\frakN$ for $\frakN$ with components
$X_1, \ldots, X_t$.  Let $n_j = |X_j|$ for $1 \le j \le t$.  Recall 
$t =  t_B$, and $t > 1$ means $\calS^{+_\frakN}_\zero(\calO) \ne 0$.
By \cref{lem:fixpt}, for a fixed $X_j$ the coefficients $\frac{|\frako_F^\times|}{|\Gamma_i|}$ 
appearing in \eqref{eq:IP} are identical for all $i$ with
$x_i \in X_j$.  Let $c_j$ be this number for $X_j$.  Then
$(\one, \one) = \sum_{i=1}^h \frac{|\frako_F^\times|}{|\Gamma_i|} = \sum_{j=1}^t c_j^{-1} n_j$.
This number is the mass $m(\calO)$ of $\calO$ studied by Eichler, and
equals $2^{1-d} |\zeta_F(-1)|N(\frakN) \prod_{p | \frakN} (1- N(\frakp)^{-1})$
(see, e.g., \cite{me:cong}).  

Let us define $\phi' \in \calM^{+_\frakN}_\zero(\calO)$ by $\phi'(x_i) = a_j$ for all
$x_i \in X_j$, where $a_j \in 2\Z + 1$ for $1 \le j \le t$.  Then $\phi' \equiv \one
\mod 2$, and the argument with the Deligne--Serre lemma above will give our
proposition if we can choose $\phi' \in \calS^{+_\frakN}_\zero(\calO)$.  By \eqref{eq:IP},
this means we want to show there is a
solution to $\sum_{j=1}^t c_j^{-1} a_j n_j = 0$ in the $a_j$'s.  We can scale the
quantities $\frac{n_j}{c_j}$ by some $\lambda \in \Q^\times$ so that
$m_j = \lambda \frac{n_j}{c_j} \in \Z$ for  $1\le j\le t$ and 
$\gcd(m_1, \ldots, m_t) = 1$.  

The hypothesis that $m(\calO)$ is even means $\sum m_j$ also is.
Writing $a_j = 2 b_j + 1$ for all $j$, our desired (scaled) linear equation is that
$\sum m_j 2 b_j = - \sum m_j$, i.e., $\sum m_j b_j = - \frac 12 \sum m_j$,  
which has a solution as $\gcd(m_1, \ldots, m_t) = 1$.
\end{proof}

Hence if the hypotheses of this proposition are satisfied, we can take
$g$ to be a newform in $S_2(\frakN)$ in \cref{cor2}.

From now on, assume $F=\Q$.  

Then the mass $m(\calO)$ is just $\frac{\phi(N)}{12}$, where
$\phi(N) = \prod_{p | N} (p-1)$, so \cref{prop:54} tells us we 
can take $g$ to be a cusp form if $8 | \phi(N)$ and $t_B > 1$.  Recall
$t_B \ge 2^{-\omega(N)} h_B$ and $h_B \ge m(\calO)$, so $t_B > 1$ whenever
$\phi(N) > 12 \cdot 2^{\omega(N)}$.  This is automatic if $N$ has at least 5
prime divisors, in which case we also have $8 | \phi(N)$.  Hence if $\omega(N) > 3$,
we may take $g$ to be a cusp form in \cref{thm2}.

Suppose $N = p_1 p_2 p_3$ with $2 < p_1 < p_2 < p_3$.  Automatically
$8 | \phi(N)$.  Also, if $p_1 \ge 5$ or $p_1 = 3$, $p_2 \ge 7$, or 
$p_1 = 3$, $p_2 = 5$, $p_3 \ge 17$ then the above reasoning shows $t_B > 1$.  
The remaining possibilities are $N = 3 \cdot 5 \cdot 7$, $N = 3 \cdot 5 \cdot 11$ or 
$N = 3 \cdot 5 \cdot 13$, and in fact one checks that $t_B > 1$ in these three cases
as well.  Hence if $N$ is an odd product of 3 primes, we can take $g$ to be a 
cusp form in \cref{thm2}.

To finish the theorems in the introduction, it thus remains treat $N$ prime.  

\begin{prop} \label{prop:55}
Let $f \in S_2^{\new, \chi}(N)$ be a newform.  Suppose there
exists $p | N$ such that $w_p(f) = +1$ but $p$ does not satisfy any of the following conditions:
\begin{enumerate}[(a)]
\item $p \equiv 7 \mod 8$ and $N$ is even; or 

\item $p \ne 2$ and ${-p \leg q} = 1$ for some odd prime $q | N$; or

\item $p=2$, $N$ is divisible by a prime which is $1 \mod 4$ as well a (not necessarily 
different) prime $q$ such that ${-2 \leg q} = 1$.
\end{enumerate}
Then there exists a newform $g \in S_2^{\new, -_N}(N)$
such that $f \equiv g \mod 2$.
\end{prop}

\begin{proof}
Let $\phi$ be an associated integral newform to $f$.
By \cref{lem:fixpt}, the hypothesis on $p$ implies $\sigma_p$ has at least one fixed point 
$x_i \in \Cl(\calO)$.  The condition $T_p \phi = - \phi$ then implies $\phi(x_i) = 0$.
Now the construction of $\phi' \in \calM_0^{+_N}(\calO)$ in \cref{mainthm}/\cref{cor2} with 
$\phi' \equiv \phi \mod 2$ also means $\phi'(x_i) = 0$.  

Write $\phi' = \sum a_j  \phi_j$ as a sum of eigenforms.
Note $a_j = 0$ unless $\phi_j \in \calM_0^{+_N}(\calO)$ since $\phi' \in \calM_0^{+_N}(\calO)$.
While we used the Deligne--Serre lemma above to get an eigenform $\phi''$ with the
same Hecke eigenvalues as $\phi'$ mod 2, we gave a different argument for this
type of result in the proof of Theorem 2.1 of \cite{me:cong}.
That argument tells us that (possibly upon replacing $\phi'$ with a different form
in $\calM_0^{+_N}(\calO)$ which is congruent to $\phi'$ mod 2) the Hecke eigenvalues of $\phi_j$ are congruent to the
Hecke eigenvalues of $\phi'$ mod 2 for all $j$ such that $a_j \ne 0$.  Say $\phi_m = \one$
is the constant function generating $\calE_2(\calO)$.  Since $\phi'(x_i) = 0 \ne \phi_1(x_i)$, 
it is not possible
that $a_j = 0$ for all $j \ne m$.  This gives (at least one) $\phi_j \in \calS_0^{+_N}(\calO)$
congruent to $\phi$ mod 2, which gives our desired $g$ by the Jacquet--Langlands 
correspondence.
\end{proof}

Let us finish by explicating additional conditions when $F=\Q$ and
$k=2$ where we can take $g$ to be a cusp form in \cref{thm2}.
Assume $N = 2 p_1p_2$ with $2 < p_1 < p_2$.  

First note that $8 | \phi(N)$ if $p_1$ or $p_2$ is $1 \mod 4$.
Here $t_B > 1$ if $p_1 \ge 11$ or $p_1 = 7$, $p_2 \ge 19$ or $p_1 = 5$, $p_2 \ge 29$ or $p_1 = 3$, $p_2 \ge 53$.  Thus $t_B > 1$ if $N > 294$ by this reasoning,
and exact calculation of class numbers shows in fact $t_B > 1$ if $N > 258$.
Hence if $N > 258$ is an even product of 3 primes,
at least one of which is $1 \mod 4$, then we can take $g$ to be a cusp form
in \cref{thm2} by \cref{prop:54}.

On the other hand, suppose $p_1 \equiv p_2 \equiv 3 \mod 4$. 
Then $p=2$ never satisfies (a), (b) or (c) of \cref{prop:55}.  Now $p_1, p_2$
never satisfy (c), and satisfy (a) if and only if they are $7 \mod 8$. 
By quadratic reciprocity, ${-p_1 \leg p_2}
= (-1) {-p_2 \leg p_2}$, so (b) of \cref{thm2} will be satisfied for exactly one of
$p=p_1$ and $p=p_2$.  Hence \cref{prop:55} cannot be used to guarantee
$g$ is a cusp form in the remaining cases of $N = 2 p_1 p_2$ for $f$ with arbitrary
signs.  However, it can be used to say we can take $g$ to be a cusp form
if $w_{p_1}(f) = w_{p_2}(f) = +1$ (or just $w_{p_i}(f) = +1$ for whichever $p_i$
does not satisfy (b)) and $p_1 \equiv p_2 \equiv 3 \mod 8$.

%
%

\end{document}